\documentclass[a4paper,12pt]{article}
\usepackage[T1]{fontenc}
\usepackage{mathrsfs}
\usepackage{graphicx}
\usepackage{enumerate}
\usepackage{multicol}
\usepackage{color}
\usepackage{amsmath,amssymb}
\usepackage{amsthm}
\usepackage{textcomp}
\usepackage{mathabx}

\usepackage{hyperref}

\usepackage{times}
\usepackage[varg]{txfonts}

\usepackage[top=1in, bottom=1in, left=0.75in, right=0.75in]{geometry}

\theoremstyle{plain}
\newtheorem{thm}{Theorem}
\newtheorem{lem}{Lemma}
\newtheorem{cor}{Corollary}
\newtheorem{prop}{Proposition}

\theoremstyle{definition}
\newtheorem{dfn}{Definition}
\newtheorem{ex}{Example}

\theoremstyle{remark}
\newtheorem{rem}{Remark}

\newtheorem*{ackn}{Acknowledgment}
\newcommand{\ph}{\varphi}

\newcommand{\diff}[2]{\dfrac{\mathrm{d}^{#1}}{\mathrm{d}{#2}^{#1}}}
\newcommand{\supp}{\mathrm{supp}\,}

\title{Second-order derivations of function spaces --- a characterization of  second-order differential operators}
\author{Włodzimierz Fechner and Eszter Gselmann}

\begin{document}

\maketitle

\begin{abstract}
Let $\Omega \subset \mathbb{R}$ be a nonempty and open set, then for all $f, g, h\in \mathscr{C}^{2}(\Omega)$ we have 
\begin{multline*}
 \diff{2}{x}(f\cdot g\cdot h) 
 -f\diff{2}{x}(g\cdot h)-g\diff{2}{x}(f\cdot h)-h\diff{2}{x}(f\cdot g)
 \\
 + f\cdot g\diff{2}{x}h+f\cdot h\diff{2}{x}g+g\cdot h\diff{2}{x}f=0
\end{multline*}
The aim of this paper is to consider the corresponding operator equation 
\[
 D(f\cdot g \cdot h) - fD(g\cdot h) - gD(f\cdot h) - hD(f \cdot g) + f\cdot g D(h) + f\cdot h D(g) +g\cdot h D(f) =0
\]
for operators $D\colon \mathscr{C}^{k}(\Omega)\to \mathscr{C}(\Omega)$, where $k$ is a given nonnegative integer and the above identity is supposed to hold for all $f, g, h \in  \mathscr{C}^{k}(\Omega)$. We show that besides the operators of first and second derivative, there are more solutions to this equation, and we characterize all solutions. Some special cases characterizing differential operators are also studied.
\end{abstract}

\section{Introduction}

Let $\Omega \subset \mathbb{R}$ be a nonempty and open set and $k$ be a nonnegative integer. Let us consider the function space 
\[
 \mathscr{C}^{k}(\Omega)= 
 \left\{  f\colon \Omega \to \mathbb{R}\, \vert \, f \text{ is } k \text{ times continuously differentiable}\right\}. 
\]
For $k=0$, instead of $\mathscr{C}^{0}(\Omega)$ we simply write $\mathscr{C}(\Omega)$ for the linear spaces of all continuous functions $f\colon \Omega\to \mathbb{R}$.

A fundamental property of the first-order derivative is the Leibniz rule, i.e. 
\[
\diff{}{x}(f\cdot g) = f \cdot \diff{}{x} g+ \diff{}{x}f \cdot g 
\qquad 
\left(f, g\in \mathscr{C}^{1}(\Omega)\right). 
\]

H.~König and V.~Milman in \cite{KonMil11}, addressed the question of to what extent the Leibniz rule characterizes the derivative. Their result is  the theorem below, cf.~also \cite[Theorem 3.1]{KonMil18} and \cite{GolSem96}. 

\begin{thm}[König-Milman]
 Let $\Omega\subset \mathbb{R}$ be a nonempty and open set and $k$ be a nonnegative integer. Suppose that the operator $T\colon \mathscr{C}^k(\Omega)\to \mathscr{C}(\Omega)$ satisfies the Leibniz rule, i.e., 
 \begin{equation}\label{LR}
  T(f\cdot g)= f\cdot T(g)+T(f)\cdot g 
  \qquad 
  \left(f, g\in \mathscr{C}^k(\Omega)\right). 
 \end{equation}
Then there exist functions $c, d\in \mathscr{C}^k(\Omega)$ such that for all $f\in \mathscr{C}(\Omega)$ and $x\in \Omega$ 
\[
 T(f)(x)= c(x)\cdot f(x) \cdot \ln \left(\left|f(x)\right|\right)+ d(x) \cdot f'(x). 
\]
For $k=0$ we necessarily have $d=0$. 
Conversely, any such map $T$ satisfies \eqref{LR}. 
\end{thm}

In case of the second-order differential operator, the so-called second-order Leibniz rule, i.e. 
 \begin{equation}\label{idfg}
  \diff{2}{x}(f\cdot g)= \diff{2}{x}f\cdot g+2\diff{}{x}f\cdot \diff{}{x}g+f\cdot \diff{2}{x}g 
  \qquad 
  \left(f, g \in \mathscr{C}^{2}(\Omega)\right)
 \end{equation}
 plays a key role. 
 Motivated by this identity, König--Milman \cite{KonMil18} studied also the corresponding operator identity 
 \[
  T(f\cdot g)= T(f)\cdot g+f\cdot T(g)+2A(f)\cdot A(g) 
  \qquad 
  \left(f, g\in \mathscr{C}^{k}(\Omega)\right)
 \]
for operators $T, A\colon \mathscr{C}^{k}(\Omega)\to \mathscr{C}(\Omega)$ . 
In this case, the precise statement is more complicated. We will need two definitions before we can formulate their result. 

\begin{dfn}
 Let $k$ be a nonnegative integer, $\Omega \subset \mathbb{R}$ be an open set. An operator $A\colon \mathscr{C}^{k}(\Omega)\to \mathscr{C}(\Omega)$ is \emph{non-degenerate} if for each nonvoid open subset $U\subset \Omega$ and all $x\in U$, there exist functions $g_{1}, g_{2}\in \mathscr{C}^{k}(\Omega)$ with supports in $U$ such that the vectors $(g_{i}(x), A g_{i}(x))\in \mathbb{R}^{2}$, $i=1, 2$ are linearly independent in $\mathbb{R}^{2}$. 
\end{dfn}

\begin{dfn}\label{dfn_nontrivial}
 Let $k\geq 2$ be a positive integer and $\Omega\subset \mathbb{R}$ be an open set. We say that the operator $A\colon \mathscr{C}^{k}(\Omega)\to \mathscr{C}(\Omega)$ \emph{depends non-trivially on the derivative} if there exists $x\in \Omega$ and there are functions $f_{1}, f_{2}\in \mathscr{C}^{k}(\Omega)$ such that 
 \[
  f_{1}(x)= f_{2}(x) \quad \text{and} \quad A f_{1}(x)\neq A f_{2}(x)
 \]
holds. 
\end{dfn}

We recall  Theorem 7.1  from \cite{KonMil18} in connection with the second-order Leibniz rule. 

\begin{thm}[König-Milman]\label{Konig-Milman}
 Let $k$ be a nonnegative integer and $\Omega \subset\mathbb{R}$ be a domain (i.e. an open and connected set). Assume that $T, A\colon \mathscr{C}^{k}(\Omega)\to \mathscr{C}(\Omega)$ satisfy 
 \begin{equation}\label{Tfg}
  T(f\cdot g)= T(f)\cdot g+f\cdot T(g)+2A(f)\cdot A(g) 
  \qquad 
  \left(f, g\in \mathscr{C}^{k}(\Omega)\right)
 \end{equation}
and that $A$ is non-degenerate and depends non-trivially on the derivative. 
Then there are continuous functions $a, b, c\in \mathscr{C}(\Omega)$ such that
\[
 \begin{array}{rcl}
  T(f)(x)&=& \frac{1}{2}c(x)^2\cdot  f''(x)+R(f)(x)\\[2mm]
  A(f)(x)&=& c(x)\cdot f'(x)
 \end{array}
\qquad 
\left(f\in \mathscr{C}^{k}(\Omega), x\in \Omega\right), 
\]
where 
\[
 R(f)(x)= b(x)\cdot  f'(x)+a(x) \cdot f(x)\ln \left(|f(x)|\right) 
 \qquad 
 \left(f\in \mathscr{C}^{k}(\Omega)\right). 
\]
If $k=1$, then necessarily $c\equiv 0$. Further, if $k=0$, then necessarily $b\equiv 0$ and $c\equiv 0$.

Conversely, these operators satisfy the second-order Leibniz rule \eqref{TA}. 
\end{thm}

\begin{cor}\label{Cor_KM}
Under the assumptions of the previous theorem, suppose that the operators $T, A\colon \allowbreak \mathscr{C}^{k}(\Omega) \allowbreak \to \mathscr{C}(\Omega)$ are additive. Then $a\equiv 0$ in the above theorem. In other words, there are continuous functions $b, c\in \mathscr{C}(\Omega)$ such that 
\begin{equation*}
 \begin{array}{rcl}
  \label{TA} T(f)(x)&=& \frac{1}{2}c(x)^{2} \cdot  f''(x)+ b(x)\cdot  f'(x) \\[2mm]
   A(f)(x)&=& c(x)\cdot  f'(x)
 \end{array}
\qquad 
\left(f\in \mathscr{C}^{k}(\Omega), x\in \Omega\right). 
\end{equation*}
If $k=1$, then necessarily $c\equiv 0$. Further, if $k=0$, then necessarily $b\equiv 0$ and $c\equiv 0$. 

Conversely, these operators are additive and satisfy the second-order Leibniz rule \eqref{TA}. 
\end{cor}

\medskip

In the present article, we propose another operator equation for the characterization of the second-order Leibniz rule, which in our opinion has some advantages comparing to \eqref{Tfg}.
Easy computation shows that 
\begin{multline}\label{id_second}
 \diff{2}{x}(f\cdot g\cdot h)
 -f\diff{2}{x}(g\cdot h)-g\diff{2}{x}(f\cdot h)-h\diff{2}{x}(f\cdot g)
 \\
 +f\cdot g\diff{2}{x}h+f\cdot h\diff{2}{x}g+g\cdot h\diff{2}{x}f=0
\end{multline}
holds for all $f, g, h\in \mathscr{C}^{2}(\Omega)$. 
 
The identity \eqref{Tfg}, has a disadvantage in some sense compared to identity 
\begin{multline*}
D(f\cdot g \cdot h) - fD(g\cdot h) - gD(f\cdot h) - hD(f \cdot g)
\\
+ f\cdot g D(h) + f\cdot h D(g) +g\cdot h D(f) =0
\qquad
\left(f, g, h\in \mathscr{C}^{k}(\Omega)\right)
\end{multline*}
Namely, the second-order Leibniz rule \eqref{Tfg} includes not only the second-order but also the first-order differential operator. In addition, a significant condition in the above theorem is that the operator $A$ is non-degenerate and depends non-trivially on the derivative. As we will see in the third section, neither these conditions, nor linearity, need not be assumed for the operator $D$ in the case of the latter identity.

\section{Operator relations and their connections}
 
Let $k$ be a fixed nonnegative integer and $\Omega\subset \mathbb{R}$ be a nonempty and open set. In what follows, we will study operators $D\colon \mathscr{C}^{k}(\Omega)\to \mathscr{C}(\Omega)$ that fulfil 
\begin{equation}\label{id_2}
D(f\cdot g \cdot h) - fD(g\cdot h) - gD(f\cdot h) - hD(f \cdot g) + f\cdot g D(h) + f\cdot h D(g) +g\cdot h D(f) =0
\end{equation}
for all $f, g, h\in \mathscr{C}^{k}(\Omega)$. We emphasize that unless written otherwise, the operator $D$ is not assumed to be linear. As we will see \eqref{id_2} turns out to be suitable for characterizing second-order linear differential operators in function spaces. 

Besides equation \eqref{id_2}, the related operator identity 
\begin{equation}\label{id_powers}
 D(f^3) - 3fD(f^2) + 3f^2D(f)=0 
 \qquad 
 \left(f\in \mathscr{C}^{k}(\Omega)\right)
\end{equation}
will also be studied for the operator $D\colon \mathscr{C}^{k}\Omega)\to \mathscr{C}(\Omega)$. Here $f^{i}$ denotes the $i$\textsuperscript{th} power of the function $f$. 
Note that if one substitute in \eqref{id_2} $g=f$ and $h=f$, then \eqref{id_2} reduces to \eqref{id_powers}. Therefore, the latter identity is more general.

\begin{rem}[Leibniz rule $\Rightarrow$ identity \eqref{id_2}]
 Assume that the operator $D\colon \mathscr{C}^{k}(\Omega)\to \mathscr{C}(\Omega)$ fulfills the Leibniz rule, i.e., we have 
 \[
  D(f\cdot g)= D(f)\cdot g+f\cdot D(g) 
  \qquad 
  \left(f, g\in \mathscr{C}^{k}(\Omega)\right). 
 \]
Then $D$ satisfies identity \eqref{id_2}, too. Indeed, by a successive application of the Leibniz rule, we derive that
\begin{align*}
 D(fgh)&= D((gf)\cdot h)= D(fg)\cdot h+fg\cdot D(h) \\
 &= \left[D(f)\cdot g+f\cdot D(g)\right] \cdot h+ fg\cdot D(h)\\
 &= D(f)\cdot gh+D(g)\cdot fh+D(h)\cdot fg
\end{align*}
holds for all $f, g, h\in \mathscr{C}^{k}(\Omega)$. Therefore 
\begin{align*}
 D(f\cdot g \cdot h)& - fD(g\cdot h) - gD(f\cdot h) - hD(f \cdot g) + f\cdot g D(h) + f\cdot h D(g) +g\cdot h D(f) \\
 & = D(f)\cdot gh+D(g)\cdot fh+D(h)\cdot fg \\ 
 &-f \cdot \left[D(g)\cdot h+g\cdot D(h)\right]
 -g \cdot \left[D(f)\cdot h+f\cdot D(h)\right]
 -h \cdot \left[D(f)\cdot g+f\cdot D(g)\right] \\
 & + f\cdot g D(h) + f\cdot h D(g) +g\cdot h D(f) =0
\end{align*}
for all $f, g, h\in \mathscr{C}^{k}(\Omega)$. 
\end{rem}

\begin{rem}[Second-order Leibniz rule $\Rightarrow$ \eqref{id_2}]
 Let $D, A\colon \mathscr{C}^{k}(\Omega)\to \mathscr{C}(\Omega)$ be operators such that the operator $A$ satisfies the Leibniz rule and the pair $(D, A)$ fulfils the second-order Leibniz rule \eqref{Tfg}. Then $D$ fulfils identity \eqref{id_2}.
 
 Indeed, if the pair $(D, A)$ satisfies the second-order Leibniz rule \eqref{Tfg}, then by a successive application, we deduce that 
 \begin{align*}
D(fgh)&= D((fg)\cdot h) \\
&= D(fg)\cdot h+fg\cdot D(h)+2A(fg)\cdot A(h)\\
&= h\cdot \left[D(f)\cdot g+f\cdot D(g)+2A(f)\cdot A(g)\right]+fg\cdot D(h)+2A(fg)\cdot A(h)\\
&= D(f)\cdot gh+D(g)\cdot fh+D(h)\cdot fg+2hA(f)\cdot A(g)+2A(fg)\cdot A(h)
\end{align*}
holds for all $f, g, h\in \mathscr{C}^{k}(\Omega)$. 
From this, we obtain
\begin{multline}\label{eq_rem2}
 D(fgh)-f\cdot D(fg)-g\cdot D(fh)-h\cdot D(fg)+fg \cdot D(h)+fh\cdot D(g)+gh\cdot D(f)
 \\
 =
 2A(h)\cdot \left[A(fg)-f\cdot A(g)-g\cdot A(f)\right]
\end{multline}
for all $f, g, h\in \mathscr{C}^{k}(\Omega)$. Thus if $A$ satisfies the Leibniz rule, then $D$ fulfils identity \eqref{id_2} for all $f, g, h\in \mathscr{C}^{k}(\Omega)$. 

Observe that equation \eqref{eq_rem2} holds without assuming $A$ to fulfil the Leibniz rule. The only thing we used there is that $D$ satisfies the second-order Leibniz rule. Therefore, if the pair $(D, A)$ satisfies the second-order Leibniz rule, then the operator $D$ fulfills identity \eqref{id_2} \emph{if and only if} 
\[
 A(h)(x)\cdot \left[A(fg)(x)-f(x)\cdot A(g)(x)-g(x)\cdot A(f)(x)\right]=0
\]
holds for all $f, g, h\in \mathscr{C}^{k}(\Omega)$ and for all $x\in \Omega$. 
This identity holds obviously when $ A(h)(x)=0$
for all $h\in \mathscr{C}^{k}(\Omega)$ and for all $x\in \Omega$. 
The other possibility is that there exists a function $h\in \mathscr{C}^{k}(\Omega)$ and there is a point $x\in \Omega$ such that $A(h)(x)\neq 0$. As $A(h)\in \mathscr{C}(\Omega)$, there exists an open interval $J\subset \Omega$, containing the point $x$ such that $A(h)\vert_{J}\neq 0$. This however yields that necessarily 
\[
 A(fg)(x)=f(x)\cdot A(g)(x)+g(x)\cdot A(f)(x)
\]
whenever $f, g\in \mathscr{C}^{k}(\Omega)$ and $x\in J$. 

Moreover, the left-hand side of equation \eqref{eq_rem2} is symmetric in $f, g$ and $h$. Thus the right-hand side must also be symmetric in all of its variables. 
So if the pair $(D, A)$ fulfils the second-order Leibniz rule \eqref{Tfg}, then 
\begin{align*}
D(fgh)&-f\cdot D(fg)-g\cdot D(fh)-h\cdot D(fg) +fg \cdot D(h)+fh\cdot D(g)+gh\cdot D(f) \\
&= \frac{2}{3}A(h)\cdot \left[A(fg)-f\cdot A(g)-g\cdot A(f)\right] 
+ \frac{2}{3}A(g)\cdot \left[A(fh)-f\cdot A(h)-h\cdot A(f)\right] \\
&+ \frac{2}{3}A(f)\cdot \left[A(gh)-g\cdot A(h)-h\cdot A(g)\right] 
\end{align*}
holds for all $f, g, h\in \mathscr{C}^{k}(\Omega)$. 
\end{rem}

\begin{rem}
 As our results below show (or it can be checked by a simple calculation), the operator $T$ defined by 
\[
 T(f)= f\ln(|f|)^{2} 
 \qquad 
 \left(f\in \mathscr{C}^{k}(\Omega)\right)
\]
fulfills \eqref{id_2}. Moreover, for the pair 
\[
 T(f)= f\ln(|f|)^{2} \qquad 
 \text{and}
 \qquad 
 A(f)= f\ln(|f|)
 \qquad 
 \left(f\in \mathscr{C}^{k}(\Omega)\right)
\]
we have 
\begin{align*}
 T(fg)& = fg\ln(|fg|)^{2} \\
 &= fg \left( \ln(|f|+ \ln(|g|))\right)^{2}\\
 &= f \cdot g\ln(|g|)^{2}+g\cdot f\ln(|f|)^2+2f\ln(|f|)\cdot g\ln(|g|)\\
 &= f\cdot T(g)+g\cdot T(g)+2A(f)\cdot A(g)
 \end{align*}
 for all $f, g\in \mathscr{C}^{k}(\Omega)$. So the pair $(T, A)$ satisfies the second-order Leibniz rule \eqref{Tfg}. 
At the same time, the pair $(T, A)$ is not included in Theorem \ref{Konig-Milman} because this pair does not satisfy the conditions of the above-mentioned theorem of König and Milman. 
If $U\subset \Omega$ is a nonempty and open set and $x\in U$, then we can choose the functions $g_{1}, g_{2}\in \mathscr{C}^{k}(\Omega)$ so that $g_{1}(x)\neq 0$, $g_{2}(x)\neq 0$ and $|g_{1}(x)|\neq |g_{2}(x)|$ hold. 
In this case however 
\[
 \det \begin{pmatrix}
       g_{1}(x) & g_{2}(x)\\
       g_{1}(x) \ln(|g_{1}(x) |) & g_{2}(x) \ln(|g_{2}(x) |)
      \end{pmatrix}
=
g_{1}(x)g_{2}(x)\left(\ln (|g_{2}(x)|)-\ln(|g_{1}(x) |)\right)\neq 0
\]
holds, i.e., the vectors $(g_{i}(x), A(g_{i})(x))\in \mathbb{R}^{2}$ for $i=1, 2$ are linearly independent. So the operator $A$ is non-degenerate. 

At the same time, if $x\in \Omega$ and $g_{1}, g_{2}\in \mathscr{C}^{k}(\Omega)$ are such that $g_{1}(x)= g_{2}(x)$, then 
\[
 A(g_{1})(x)= g_{1}(x)\ln(|g_{1}(x)|)= g_{2}(x)\ln(|g_{2}(x)|) = A(g_{2})(x).
\]
 Thus the operator $A$ does not fulfill the requirements of Definition \ref{dfn_nontrivial}, since $A$ depends trivially on the derivative. 
\end{rem}

\section{Results}

At first, we show that any solution $D\colon \mathscr{C}^{k}(\Omega)\to \mathscr{C}(\Omega)$ is localized on intervals.

\begin{lem}\label{lem_loc_interval}
 Let $k$ be a fixed nonnegative integer and $\Omega\subset \mathbb{R}$ be a nonempty and open set.
 Suppose that the operator  $D\colon \mathscr{C}^{k}(\Omega)\to \mathscr{C}(\Omega)$ satisfies \eqref{id_2}, then $D(\mathbf{1})=0$ and $D(-\mathbf{1})=0$ hold. Further, $D$ is \emph{localized on intervals}, that is if $J\subset \Omega$ is open and $f_{1}, f_{2}\in \mathscr{C}^{k}(\Omega)$ are functions such that $f_{1}\vert_{J}= f_{2}\vert_{J}$, then $D(f_{1})\vert_{J}= D(f_{2})\vert_{J}$. 
\end{lem}

\begin{proof}
 Equation \eqref{id_2} with the substitution $f=g=h= \mathbf{1}$ implies immediately $D(\mathbf{1})=0$. Similarly, substitution $f=g=h = -\mathbf{1}$ leads to  $D(-\mathbf{1})=0$. 

Finally, let $J\subset \Omega$ be an open set. Assume that $\varphi_{1}, \varphi_{2}\in \mathscr{C}^{k}(\Omega)$ are such that $\varphi_{1}\vert_{J}= \varphi_{2}\vert_{J}$. 
Let $x\in J$ be arbitrarily fixed and let us chose a function $\psi\in \mathscr{C}^{k}(\Omega)$ with $\psi(x)=1$ and $\supp \psi\subset J$. Then we have 
\[
 \varphi_{1}\cdot \psi= \varphi_{2}\cdot \psi. 
\]
If we substitute at first 
\[
 f= \varphi_{1} \quad \text{ and } \quad  g=h= \psi 
\]
in equation \eqref{id_2} and secondly 
\[
 f= \varphi_{2} \quad \text{ and } \quad   g=h = \psi, 
\]
then we arrive at
\[
D(\varphi_i \cdot\psi^2) - \varphi_i D(\psi^2) - 2 \psi D(\varphi_i\cdot\psi) + 2 \varphi_i \cdot\psi D(\psi) + \psi^2 D(\varphi_i) = 0
\]
for $i = 1, 2$. Note that $D(\varphi_{1}\cdot \psi^{j})= D(\varphi_{2}\cdot \psi^{j})$  and moreover, equality $D(\mathbf{1})=0$ implies that $D(\psi^j)(x)=0$ for $j\in \{1, 2\}$. Therefore the equality
$D(\varphi_{1})(x)= D(\varphi_{2})(x)$ follows. 
Since $x\in J$ was arbitrary, we obtain finally that $D(\varphi_{1})\vert_{J}= D(\varphi_{2})\vert_{J}$. Thus $D$ is localized on intervals. 
\end{proof}

In view of \cite[Proposition 3.3]{KonMil18}, which is the statement below, any operator that is localized on intervals is pointwise localized. 

\begin{prop}\label{prop_loc_pointwise}
 Let $k$ be a fixed nonnegative integer, $\Omega\subset \mathbb{R}$ be an open set and suppose that the operator $T\colon \mathscr{C}^{k}(\Omega)\to \mathscr{C}(\Omega)$ is localized on intervals, then there exists a function $F\colon \Omega\times \mathbb{R}^{k+1}\to \mathbb{R}$ such that 
 \[
  (Tf)(x)= F(x, f(x), f'(x), \ldots, f^{(k)}(x))
 \]
for all $x\in \Omega$ and for all $f\in \mathscr{C}^{k}(\Omega)$. 
\end{prop}

During the proof of our main results, the following consequence of the Fa\`{a} di Bruno formula will be utilized. 

\begin{lem}\label{lem_faa}
 Let $k, l$ be positive integers, $l\leq k$, $\Omega\subset \mathbb{R}$ be a nonempty and open set and $f\in \mathscr{C}^{k}(\Omega)$ be a positive function. Then 
 \begin{multline*}
  \diff{l}{x}\ln \circ f(x)
  \\
  = 
  \sum_{\substack{\sum_{i=1}^{l}i m_{i}=l}}\frac{l!}{m_{1}!\cdots m_{l}!} \cdot 
  \frac{(-1)^{m_{1}+\cdots+m_{l}-1}(m_{1}+\cdots+m_{l}-1)!}{f(x)^{m_{1}+\cdots+ m_{l}}} \cdot \prod_{1\leq j\leq  l} \left(\frac{f^{(j)}(x)}{j!}\right)^{m_{j}} 
  \\ 
  \left(x\in \Omega\right). 
\end{multline*}
\end{lem}

In the proof of Theorem \ref{thm_main}, we will show that the operator equation \eqref{id_2} can be reduced to a so-called Aichinger type functional equation. To solve the latter, we need the following results from \cite{Alm23}, see also \cite{AihMoo21} and \cite{Alm23b}. However, before presenting them, we need to recall the notion of generalized polynomials from \cite{Sze91}. 

\begin{dfn}
 Let $G, S$ be commutative semigroups, $n\in \mathbb{N}$ and let $A\colon G^{n}\to S$ be a function.
 We say that $A$ is \emph{$n$-additive} if it is a homomorphism of $G$ into $S$ in each variable.
 If $n=1$ or $n=2$ then the function $A$ is simply termed to be \emph{additive}
 or \emph{bi-additive}, respectively.
\end{dfn}

The \emph{diagonalization} or \emph{trace} of an $n$-additive
function $A\colon G^{n}\to S$ is defined as
 \[
  A^{\ast}(x)=A\left(x, \ldots, x\right)
  \qquad
  \left(x\in G\right).
 \]
 
 \begin{dfn}
 Let $G$ and $S$ be commutative semigroups, a function $p\colon G\to S$ is called a \emph{generalized polynomial} from $G$ to $S$ if it has a representation as the sum of diagonalizations of symmetric multi-additive functions from $G$ to $S$. In other words, a function $p\colon G\to S$ is a 
 generalized polynomial if and only if, it has a representation 
 \[
  p= \sum_{k=0}^{n}A^{\ast}_{k}, 
 \]
where $n$ is a nonnegative integer and $A_{k}\colon G^{k}\to S$ is a symmetric, $k$-additive function for each 
$k=0, 1, \ldots, n$. In this case we also say that $p$ is a generalized polynomial \emph{of degree at most $n$}. 

Let $n$ be a nonnegative integer, functions $p_{n}\colon G\to S$ of the form 
\[
 p_{n}= A_{n}^{\ast}, 
\]
where $A_{n}\colon G^{n}\to S$ is symmetric and  $n$-additive function are the so-called \emph{generalized monomials of degree $n$}. 
\end{dfn}

\begin{rem}
 Generalized monomials of degree $0$ are constant functions, generalized monomials of degree $1$ are \emph{additive} functions. 
 Furthermore, generalized monomials of degree $2$ are called \emph{quadratic} functions. 
 \end{rem}

In the following, we rely on several results from \cite{Alm23, Alm23b}, whose statements we reproduce here for completeness.

\begin{lem}\label{lem_aichinger}
Let $S$ be a commutative cancellative semigroup and $H$ be a commutative group. Let $G=S-S$
be a natural extension of $S$. Assume that multiplication by $m!$ is bijective on $H$. Let $f\colon S\to H$ be a solution of 
\[
 f(x_{1}+\cdots+x_{m+1})= \sum_{i=1}^{m+1} g_{i}(x_{1}, \ldots, x_{i-1}, \hat{x_{i}}, x_{i-1}, \ldots, x_{m+1}) 
 \qquad 
 \left(x_{1}, \ldots, x_{m+1}\in S\right)
\]
with certain functions $g_{i}\colon S^{m}\to H$, where for all $i=1, \ldots, m+1$, the symbol $\hat{x_{i}}$ denotes that the function $g_{i}$ does not depend on the variable $x_{i}$. Then there exists a generalized polynomial $F\colon G\to H$ of degree at most $m$, such that $F\vert_{S}=f$ and the functions $g_{i}$ are also generalized polynomials of degree at most $m$. 
\end{lem}

\begin{cor}\label{cor_aichinger}
 Let $k$ and $m$ be nonnegative integers and $f\colon \mathbb{R}^{k+1}\to \mathbb{R}$  be a function such that 
 \[
 f(x_{1}+\cdots+x_{m+1})= \sum_{i=1}^{m+1} g_{i}(x_{1}, \ldots, x_{i-1}, \hat{x_{i}}, x_{i-1}, \ldots, x_{m+1}) 
 \qquad 
 \left(x_{1}, \ldots, x_{m+1}\in \mathbb{R}^{k+1}\right)
\]
holds with certain functions $g_{i}\colon \mathbb{R}^{m(k+1)}\to \mathbb{R}$. If the function $f$ is continuous at a point $x^{\ast}\in \mathbb{R}^{k+1}$, then $f$ is a $(k+1)$-variable ordinary polynomial of degree at most $m$. 
\end{cor}

\begin{thm}\label{thm_main}
 Let $k$ be a fixed nonnegative integer and $\Omega\subset \mathbb{R}$ be a nonempty and open set.
 Suppose that the operator  $D\colon \mathscr{C}^{k}(\Omega)\to \mathscr{C}(\Omega)$ satisfies \eqref{id_2}
 for all $f, g, h\in \mathscr{C}^{k}(\Omega)$. 
 Then there exist functions $c_{0}, c_{1}, c_{2}, d_{00}\in \mathscr{C}^{k}(\Omega)$ such that 
\begin{equation}
 D(f)(x)= c_{0}(x)f(x)\ln(|f(x)|)+c_{1}(x)f'(x)+c_{2}(x)f''(x)+d_{00}(x)f(x)\left(\ln(|f(x)|)\right)^{2}
\label{D}
\end{equation} 
holds for all $x\in \Omega$ and  $f\in \mathscr{C}^{k}(\Omega)$. 
Further, 
\begin{enumerate}[(i)]
 \item if $k=0$, then $c_{1}=c_{2}=0$
 \item if $k=1$, then $c_{2}=0$. 
\end{enumerate}
Conversely, the operator $D$ given by the formula \eqref{D} satisfies \eqref{id_2}. 
\end{thm}

\begin{proof}
If $D\colon \mathscr{C}^{k}(\Omega)\to \mathscr{C}(\Omega)$ satisfies equation \eqref{id_2}, then by Lemma \ref{lem_loc_interval} and Proposition \ref{prop_loc_pointwise}, there exists a function $F\colon \Omega\times \mathbb{R}^{k+1}\to \mathbb{R}$ such that 
\[
 D(f)(x)= F(x, f(x), \ldots, f^{(k)}(x))
\]
for all $x\in \Omega$ and for all $f\in \mathscr{C}^{k}(\Omega)$. 

Consider the operator $P\colon \mathscr{C}^{k}(\Omega)\to \mathscr{C}(\Omega)$ defined by 
\[
 P(f)(x)= \frac{D(\exp\circ f)(x)}{\exp\circ f(x)} 
 \qquad 
 \left(x\in \Omega, f\in \mathscr{C}^{k}(\Omega)\right). 
\]
Since $D$ is pointwise localized, so is the operator $P$. Thus there exists a function $G\colon\mathbb{R}^{k+1}\times \Omega \to \mathscr{C}(\Omega)$ such that 
\[
 P(f)(x)= G(x, f(x), \ldots, f^{(k)}(x)) 
 \qquad 
 \left(x\in \Omega, f\in \mathscr{C}^{k}(\Omega)\right). 
\]
Let $f, g, h\in \mathscr{C}^{k}(\Omega)$ be arbitrary, then by \eqref{id_2}
\begin{align*}
 P(f+ g & + h) - P(g+ h) - P(f+ h) - P(f + g) + P(f) + P(g) +P(h)\\&=
 \frac{1}{\exp(f+g + h)}D(\exp(f + g + h)) - \frac{1}{\exp(g + h)}D(\exp(g + h)) \\&- \frac{1}{\exp(f + h)}D(\exp(f + h)) - \frac{1}{\exp(f + g)}D(\exp(f + g)) 
\\&+ \frac{1}{\exp(f)}D(\exp(f)) + \frac{1}{\exp (g) }D(\exp( g)) +\frac{1}{\exp( h)}D(\exp( h))\\&= \frac{1}{\exp(f)\cdot\exp(g)\cdot\exp(h)}\cdot \Big[D(\exp(f)\cdot\exp(g)\cdot\exp(h)) 
- \exp(f)D(\exp(g)\cdot\exp(h)) \\&- \exp(g)D(\exp(f)\cdot\exp(h)) -\exp(h)D(\exp(f)\cdot\exp(g)) \\&+\exp(f)\cdot \exp(g)D(\exp(h)) 
+ \exp(f)\cdot \exp(h)D(\exp(g)) 
+ \exp(g)\cdot \exp(h)D(\exp(f)) \Big]= 0. 
\end{align*}
For any $v_{j}= (v_{j}^{l})_{l=0}^{k}\in \mathbb{R}^{k+1}$, $j=1, 2, 3$ and $x\in \Omega$ we can find smooth functions $g_{1}, g_{2}, g_{3}\in \mathscr{C}^{k}(\Omega)$ such that 
\[
 g_{j}^{(l)}(x)= v_{j}^{l}, \qquad (l=0, 1, \dots , k, \, j=1, 2, 3 ). 
\]
From this and from the previous equality we get that the function $G$ satisfies
\[
 G(x, v_1+v_2+v_3) - G(x, v_2+v_3) - G(x, v_1 +v_3)- G(x, v_1+v_2) + G(x, v_1) + G(x, v_2)+ G(x, v_3) =0
\]
for all $x\in \Omega$ and for all $v_{1}, v_{2}, v_3\in \mathbb{R}^{k+1}$. 
In other words, for every fixed $x\in \Omega$, the mapping 
\[
 \mathbb{R}^{k+1} \ni v \longmapsto G(x, v)
\]
satisfies an Aichinger-type functional equation. 
At the same time, for every fixed $g\in \mathscr{C}^{k}(\Omega)$, the function $P(g)$ is continuous on $\Omega$. Thus for any $g\in \mathscr{C}^{k}(\Omega)$, 
the function $x\longmapsto G(x, g(x), \ldots, g^{(k)}(x))$ is continuous on $\Omega$, as well. Therefore, by Corollary \ref{cor_aichinger},  there exist continuous functions $c\colon \Omega \to \mathbb{R}^{k+1}$, $ d\colon \Omega \to \mathscr{M}_{(k+1)\times (k+1)}(\mathbb{R})$ such that 
\[
 G(x, v)=  \langle c(x) + d(x)\cdot v ,v     \rangle 
 \qquad 
 \left(x\in \Omega, v\in \mathbb{R}^{k+1}\right),
\]
where $\cdot$ denotes the matrix multiplication. 
Let us denote $c = (c_0, \dots , c_{k})$ and $d=(d_{ij})_{i , j = 0}^{k}$. Thus $c_i, d_{ij}\in \mathscr{C}(\Omega)$ for all $i, j=0, \ldots, k$. 
Using these functions, we can write the map $G$ as follows:
\[
G(x, v)= \sum_{i=0}^{k}c_{i}(x)v_{i}+\sum_{i=0}^{k}\sum_{j=0}^{k}d_{i j}(x)v_{i}v_{j} \qquad 
 \left(x\in \Omega, v = (v_0, v_1, \dots , v_k) \in \mathbb{R}^{k+1}\right).
\]
Further, without the loss of generality, we can assume that $d_{ij} = 0$ for $i<j$ (it is enough to replace $d_{ij}$ by $d_{ij}+d_{ji}$ and put $d_{ji}=0$ for $i> j$). So finally we obtain 
\[
G(x, v)= \sum_{i=0}^{k}c_{i}(x)v_{i}+\sum_{i=0}^{k}\sum_{j=i}^{k}d_{i j}(x)v_{i}v_{j} \qquad 
 \left(x\in \Omega, v = (v_0, v_1, \dots , v_k) \in \mathbb{R}^{k+1}\right).
\]
Let now $f\in \mathscr{C}^{k}(\Omega)$ be a \emph{positive} function. Then there exists a function $g\in \mathscr{C}^{k}(\Omega)$ with $f= \exp \circ g$. So we have 
\begin{align*}
 D(f)(x)&= f(x)\cdot P(\ln\circ f)(x) = f(x)\cdot G(x, (\ln\circ f) (x), (\ln \circ f)'(x), \dots , (\ln \circ f)^{(k)}(x) )
 \\
 &= f(x) \cdot \left[\sum_{i=0}^k c_i(x)(\ln \circ f)^{(i)}(x) + \sum_{i=0}^k\sum_{j=i}^k d_{ij}(x)(\ln \circ f)^{(i)}(x)(\ln \circ f)^{(j)}(x) \right].
\end{align*}
If $f\in \mathscr{C}^{k}(\Omega)$ and $f(x)<0$ holds at a point $x\in \Omega$, then there exists an open interval $J\subset \Omega$ containing the point $x$ such that $f\vert_{J}< 0$. Thus, we can find $g\in \mathscr{C}^{k}(\Omega)$ such that $g(x)<0$ for all $x\in \Omega$ and $f\vert_{J}= g\vert_{J}$. Thus Lemma \ref{lem_loc_interval} yields that $D(f)(x)= D(g)(x)$. Therefore, without loss of generality, we can assume that $f(x)<0$ for all $x\in \Omega$. Then $f=-|f|$ holds and if we substitute 
\[
 f= |f| 
 \qquad 
 g=h= -\mathbf{1} 
 \qquad 
\]
in \eqref{id_2}, then due to Lemma \ref{lem_loc_interval}, we obtain $D(-|f|)=-D(|f|)$. Thus we have 
\[
 D(f)(x) = f(x) \cdot \left[\sum_{i=0}^k c_i(x)(\ln |f|)^{(i)}(x) + \sum_{i=0}^k\sum_{j=i}^k d_{ij}(x)(\ln|f|)^{(i)}(x)(\ln|f|)^{(j)}(x) \right]
\]
for all $f\in \mathscr{C}^{k}(\Omega)$ with $f(x)\neq 0\; (x\in \Omega)$.

At the same time, the domain of the operator $D$ is $\mathscr{C}^{k}(\Omega)$ yielding that the above formula should also hold for functions that have zeros in their range, too. Using this, we will show that $c_i=0$ for $i\geq 3$ and $d_{ij}=0$ for $i+j >2$. 

In view of Lemma \ref{lem_faa}, the order is singularity of the function $f\cdot (\ln\circ |f|)^{(i)}$ is $O\left(\frac{1}{f^{i-1}}\right)$. 

Computing the derivative in case $i=2$ in the first sum, 
we get
\[
 c_{2}(x)f(x)\cdot \ln (f(x))''= 
 c_{2}(x)f(x)\cdot \left[\frac{f''(x)}{f(x)}-\frac{(f'(x))^{2}}{f(x)^{2}}\right]
 =
 c_{2}(x)f''(x)-c_{2}(x)\frac{(f'(x))^{2}}{f(x)}. 
\]
The order of singularity $O(1/f)$ of the second term should be cancelled by some other term that appears in the above formula. This is possible only in the case of the second sum above with $i=1$ and $j=1$, i.e. in the case of the term 
\[
 d_{11}(x)f(x)\cdot \ln (f(x))'\cdot \ln (f(x))'= d_{11}(x)\dfrac{(f'(x))^{2}}{f(x)}, 
\]
showing that $c_{2}=-d_{11}$. 
Similarly, 
\begin{align*}
 c_{3}(x)f(x)\cdot \ln (f(x))'''&= \\
 &= c_{3}(x)f(x)\cdot \left[\frac{f'''(x)}{f(x)}-3\frac{f'(x)f''(x)}{f(x)^2}+2\frac{(f'(x))^3}{f(x)^{3}}\right] \\
 &= c_{3}(x)f'''(x)-3c_{3}(x)\frac{f'(x)f''(x)}{f(x)}+2c_{3}(x)\frac{(f'(x))^3}{f(x)^{2}}
\end{align*}
Again the orders of singularity $O(1/f)$ and $O(1/f^2)$, resp. of the second and third terms should be cancelled by some other terms that appear in the above formula. This is possible only in the case of the second sum above with $i=1$ and $j=2$, i.e. in the case of the term 
\[
 d_{12}(x)f(x)\cdot \ln (f(x))''\cdot \ln (f(x))'
 =
 d_{12}(x)\frac{f'(x)f''(x)}{f(x)}-d_{12}(x)\frac{(f'(x))^3}{f(x)^{2}}. 
\]
This shows that we have 
\[
 d_{12}(x)-3c_{3}(x)=0 \qquad \text{and} \qquad  -d_{12}(x)+2c_{3}(x)=0, 
\]
so $c_{3}=0$ and thus also $d_{12}=0$. 
For $i\geq 4$, the term $c_{i}(x)f(x)\cdot \ln (f(x))^{(i)}$ contains a summand of the form $\dfrac{(f'(x))^{2}f^{(i-2)}(x)}{f(x)^2}$, but the second sum in the above formula does contains a summand of this form. From this, we get that necessarily $c_{i}=0$ for $i\geq 4$, as well. 

This means that we have 
\begin{multline*}
 D(f)(x)= c_{0}(x)f(x)\ln(|f(x)|)+c_{1}(x)f'(x)+c_{2}(x)f''(x)
 \\
 +d_{00}(x)f(x)\left(\ln(|f(x)|)\right)^{2}+d_{01}(x)\ln(|f(x)|)f'(x)
\end{multline*}
holds. Since $D(f)\in \mathscr{C}(\Omega)$, then $d_{0, 1}=0$. 

From this, we finally get that 
\[
 D(f)(x)= c_{0}(x)f(x)\ln(|f(x)|)+c_{1}(x)f'(x)+c_{2}(x)f''(x)+d_{00}(x)f(x)\left(\ln(|f(x)|)\right)^{2}
\]
 for all $f\in \mathscr{C}^{k}(\Omega)$ and for all $x\in \Omega$ with some appropriate functions $c_{0}, c_{1}, c_{2}, d_{00}\in \mathscr{C}(\Omega)$. 
\end{proof}

\begin{rem}
 Since $\lim_{x\to 0}x\ln(|x|)=0$ and $\lim_{x\to 0}x\ln(|x|)^{2}=0$, in the above formula we adopt  $0\ln(0)=0$ and $0\ln(0)^2=0$. 
\end{rem}

\begin{rem}
 Even though for $k\geq 3$ we have $\mathscr{C}^{k}(\Omega)\subsetneq\mathscr{C}^{2}(\Omega)$, for $k\geq 3$ there are not more solutions than for $k=2$. This means, that any solution $D\colon \mathscr{C}^{k}(\Omega)\to \mathscr{C}(\Omega)$ naturally extends (by the same formula) to an operator $\bar{D}\colon \mathscr{C}^{2}(\Omega)\to \mathscr{C}(\Omega)$. This also shows that the natural domain for equation \eqref{id_2} is the function space $\mathscr{C}^{2}(\Omega)$. 
\end{rem}

\begin{rem}
 In this remark let us limit ourselves to the case $\Omega=\mathbb{R}$. For $h\in \mathbb{R}$, let us consider the shift operators $\tau_{h}$ defined through 
 \[
  (\tau_{h}f)(x)= f(x+h) 
  \qquad 
  \left(f\in \mathscr{C}^{k}(\mathbb{R}), x\in \mathbb{R}\right). 
 \]
An operator $T\colon \mathscr{C}^{k}(\mathbb{R})\to \mathscr{C}(\mathbb{R})$ is said to be \emph{isotropic} if it commutes with all shift operators, i.e. we have $T\tau_{h}=\tau_{h}T$. 

If, in addition to the conditions of Theorem \ref{thm_main}, the operator $D$ is also isotropic, then the functions $c_{0}, c_{1}, c_{2}, d_{00}$ are constant functions, in other words, in this case we have 
\[
 D(f)(x)= c_{0}f(x)\ln(|f(x)|)+c_{1}f'(x)+c_{2}f''(x)+d_{00}f(x)\left(\ln(|f(x)|)\right)^{2}
\]
 for all $f\in \mathscr{C}^{k}(\Omega)$ and for all $x\in \Omega$ with some constants $c_{0}, c_{1}, c_{2}, d_{00}\in \mathbb{R}$. 
\end{rem}

Observe that in the first part of the proof of Theorem \ref{thm_main}, we restricted ourselves to \emph{positive} (and later to negative) functions only. The fact that the operator $D$ is defined on the entire space $\mathscr{C}^{k}(\Omega)$ was only used in the second part of the proof. Thus, the above proof shows that the operator equation \eqref{id_2} has many more solutions among \emph{positive} functions than those stated in Theorem \ref{thm_main}. The following corollary pertains to this observation. 

\begin{cor}
 Let $k$ be a fixed nonnegative integer and $\Omega\subset \mathbb{R}$ be a nonempty and open set.
 Suppose that the operator  $D\colon \mathscr{C}^{k}(\Omega)\to \mathscr{C}(\Omega)$ satisfies \eqref{id_2}
 for all $f, g, h\in \mathscr{C}^{k}(\Omega)$.  Then there are functions $c_{i}, d_{i, j}\in \mathscr{C}(\Omega)$ for all $i, j=0, \ldots, k$ such that for every \emph{positive} function $f\in \mathscr{C}^{k}\left(\Omega\right)$ and for all $x\in \Omega$ we have 
 \[
 D(f)(x)
 = f(x) \cdot \left[\sum_{i=0}^k c_i(x)(\ln \circ f)^{(i)}(x) + \sum_{i=0}^k\sum_{j=0}^k d_{i, j}(x)(\ln \circ f)^{(i)}(x)(\ln \circ f)^{(j)}(x) \right].
\]
And also conversely, the operator $D$ defined through this formula for positive functions of $\mathscr{C}^{k}(\Omega)$, satisfies operator identity \eqref{id_2} for all positive  $f, g, h\in \mathscr{C}^{k}(\Omega)$. 
\end{cor}

As an immediate consequence of Theorem \ref{thm_main},
first we describe the linear solutions of the operator equation \eqref{id_2}, and then those that annihilate polynomials of degree at most one.

\begin{cor}\label{cor_linear}
 Let $k$ be a fixed nonnegative integer and $\Omega\subset \mathbb{R}$ be a nonempty and open set.
 Suppose that the \emph{linear} operator  $D\colon \mathscr{C}^{k}(\Omega)\to \mathscr{C}(\Omega)$ satisfies \eqref{id_2}
 for all $f, g, h\in \mathscr{C}^{k}(\Omega)$. 
 Then there exist functions $c_{1}, c_{2}\in \mathscr{C}^{k}(\Omega)$ such that 
 \begin{equation}
 D(f)(x)= c_{1}(x)f'(x)+c_{2}(x)f''(x)
 \label{D2}
 \end{equation}

holds for all $f\in \mathscr{C}^{k}(\Omega)$. 
Further, 
\begin{enumerate}[(i)]
 \item if $k=0$, then $c_{1}=c_{2}=0$
 \item if $k=1$, then $c_{2}=0$. 
\end{enumerate}

Conversely, the operator $D$ given by the formula \eqref{D2} is linear and satisfies \eqref{id_2}. 
\end{cor}

\begin{cor}
 Let $k$ be a fixed nonnegative integer and $\Omega\subset \mathbb{R}$ be a nonempty and open set.
 Suppose that the operator  $D\colon \mathscr{C}^{k}(\Omega)\to \mathscr{C}(\Omega)$ satisfies \eqref{id_2}
 for all $f, g, h\in \mathscr{C}^{k}(\Omega)$ and $D$ annihilates all polynomials of degree at most $1$. 
 Then there exists a function $c_{2}\in \mathscr{C}^{k}(\Omega)$ such that 
 \begin{equation}
 D(f)(x)= c_{2}(x)f''(x)
 \label{D3}
 \end{equation}
holds for all $f\in \mathscr{C}^{k}(\Omega)$. 
Further, $k\in \left\{ 0, 1\right\}$, then $c_{2}=0$. 

Conversely, the operator $D$ given by the formula \eqref{D3} annihilates all polynomials of degree at most $1$ and satisfies \eqref{id_2}. 
\end{cor}

\begin{rem}
 In the proof of Theorem \ref{thm_main}, the condition that the range of the operator $D$ is the function space $\mathscr{C}(\Omega)$ first became relevant when we reduced the operator equation \eqref{id_2} to an Aichinger-type functional equation. There was the first point where we used the fact that it is sufficient to determine the continuous solutions of this functional equation. Naturally, the question arises as to what we can say about the solutions of the operator equation \eqref{id_2} whose range is a broader function space. Let $\mathscr{F}(\Omega)$ denote the linear space of all functions $f\colon \Omega\to \mathbb{R}$ and assume that $D\colon \mathscr{C}^{k}(\Omega) \to \mathscr{F}(\Omega)$. 
 In view of Lemma \ref{lem_aichinger}, we can only claim that the function $G$ appearing in the proof of Theorem \ref{thm_main} is a generalized polynomial of degree at most two. 
 Especially, if $A_{1}, A_{2}, A_{3}\colon \mathbb{R}\to \mathbb{R}$ are additive functions, 
 $Q\colon \mathbb{R}\to \mathbb{R}$ is a quadratic function, and $c_{0}, c_{1}, c_{2}, d_{00}\in \mathscr{F}(\Omega)$,  then the operator $D\colon \mathscr{C}^{k}(\Omega)\to \mathscr{F}(\Omega)$ defined by 
 \[
  D(f)(x)= 
  c_{0}(x)A_{0}\left(f(x)\ln(f(x))\right)+c_{1}(x)A_{1}\left(f'(x)\right)+c_{2}(x)A_{2}\left(f''(x)\right)+d_{00}(x)Q\left(f(x)\left(\ln(f(x))\right)\right)
 \]
for all $f\in \mathscr{C}^{k}(\Omega)$ and $x\in \Omega$, satisfies \eqref{id_2}. 
\end{rem}

A natural question is whether one can replace an operator relation containing more than one mapping by a relation with just one map. Similar questions have been discussed recently in 
 \cite{FecGseSwi25} and \cite{FecSwi23} in connection with other equations. It turns out that when we assume the linearity of the operator in question, then such weakening of the remaining assumption is possible. Before we proceed, we will need some auxiliary notions.

If $T\colon \mathscr{C}^{k}(\Omega)\to \mathscr{C}(\Omega)$ is an operator and $h\in \mathscr{C}^{k}(\Omega)$ is a function, then the action of the difference operator $\Delta_{h}$ on $T$ is defined as 
\[
 \Delta_{h}T(f)= T(f+h)-T(f) 
 \qquad 
 \left(f\in \mathscr{C}^{k}(\Omega)\right). 
\]
If $n$ is a positive integer and $h_{1}, \ldots, h_{n}\in \mathscr{C}^{k}(\Omega)$, then the higher order difference $\Delta_{h_{1}, \ldots, h_{n}}$ is defined as 
\[
 \Delta_{h_{1}, h_{2} \ldots, h_{n}}T= \Delta_{h_{1}}\left(\Delta_{h_{2}}\left(\cdots \Delta_{h_{n}}T\right)\right). 
\]

Using the above notions we will deduce a strengthening of Corollary \ref{cor_linear}.

\begin{prop}
 Let $k$ be a fixed nonnegative integer and $\Omega\subset \mathbb{R}$ be a nonempty and open set.
 Suppose that the \emph{linear} operator  $D\colon \mathscr{C}^{k}(\Omega)\to \mathscr{C}(\Omega)$ satisfies 
 \begin{equation}\label{id_single}
  D(f^3) - 3fD(f^2) + 3f^2D(f)=0 
 \qquad 
 \left(f\in \mathscr{C}^{k}(\Omega)\right)
 \end{equation}
Then there exist functions $c_{1}, c_{2}\in \mathscr{C}^{k}(\Omega)$ such that $D$ is of the form \eqref{D2}
for all $f\in \mathscr{C}^{k}(\Omega)$. 
Further, 
\begin{enumerate}[(i)]
 \item if $k=0$, then $c_{1}=c_{2}=0$
 \item if $k=1$, then $c_{2}=0$. 
\end{enumerate}
\end{prop}
\begin{proof}
 We show that if the linear operator $D\colon \mathscr{C}^{k}(\Omega)\to \mathscr{C}(\Omega)$ satisfies equation \eqref{id_single}, then it also fulfils equation \eqref{id_2} and then Corollary \ref{cor_linear} can be applied to deduce the statement. 
 Let us consider the operator $\mathcal{A}_{3}\colon \mathscr{C}^{k}(\Omega)\times \mathscr{C}^{k}(\Omega) \times \mathscr{C}^{k}(\Omega)\to \mathscr{C}(\Omega)$ defined through 
 \begin{multline*}
  \mathcal{A}_{3}(f, g, h)= 
  D(f\cdot g \cdot h) - fD(g\cdot h) - gD(f\cdot h) - hD(f \cdot g)
  \\
  + f\cdot g D(h) + f\cdot h D(g) +g\cdot h D(f)
  \qquad 
  \left(f, g, h\in \mathscr{C}^{k}(\Omega)\right). 
 \end{multline*}
Due to equation \eqref{id_single} we have 
\[
 \mathcal{A}_{3}(f, f, f)= D(f^3) - 3fD(f^2) + 3f^2D(f)=0
\]
for all $f\in \mathscr{C}^{k}(\Omega)$. 
As $D$ is a linear operator, the operator $\mathcal{A}_{3}$ is linear in each of its variables. Therefore, on one hand, all the differences of the mapping 
\[
 \mathscr{C}^{k}(\Omega)\ni l \longmapsto  \mathcal{A}_{3}(l, l, l)
\]
are identically zero. So especially, we have 
\[
 \Delta_{f, g, h}\mathcal{A}_{3}(0, 0, 0)= 0 
 \qquad 
 \left(f, g, h\in \mathscr{C}^{k}(\Omega)\right). 
\]
On the other hand, since $\mathcal{A}_{3}$ is linear in all of this variables, we have 
\[
 \Delta_{f, g, h}\mathcal{A}_{3}(0, 0, 0)= 6! \cdot \mathcal{A}_{3}(f, g, h)
 \qquad 
 \left(f, g, h\in \mathscr{C}^{k}(\Omega)\right). 
\]
Consequently, 
\[
 \mathcal{A}_{3}(f, g, h)=0
\]
for all $f, g, h\in \mathscr{C}^{k}(\Omega)$. Thus the linear operator $D$ satisfies identity \eqref{id_2} and Corollary \ref{cor_linear} ends the proof. 
\end{proof}

\begin{rem}
 While proving the previous statement, the linearity of the operator $D$ was essential. 
 As identity \eqref{id_single} results from \eqref{id_2} by taking $g=f$ and $h=f$, we obtain that if $c, d\in \mathscr{C}^{k}(\Omega)$, then the operator $R\colon \mathscr{C}^{k}(\Omega)\to \mathscr{C}(\Omega)$ defined via 
 \[
  R(f)(x)= c(x)\cdot f(x)\ln(|f(x)|)+d(x)\cdot f(x)\ln(|f(x)|)^{2} 
  \qquad 
  \left(f\in \mathscr{C}^{k}(\Omega), x\in \Omega\right)
 \]
is a nonlinear operator that solves identity \eqref{id_single}. This shows in particular that identity \eqref{id_single} has solutions among nonlinear operators. At the same time $R$ solves identity \eqref{id_2}, too. 
\end{rem}

In our last example we will address the question whether there is a nonlinear operator that solves identity \eqref{id_single} but does not satisfy identity \eqref{id_2}.

\begin{ex}
Assume that $\ph\colon ]1, +\infty[ \to \mathbb{R}$ is a mapping that satisfies the equation
\begin{equation}
\ph(3x) = 3 \ph(2x)-3\ph(x), \quad (x \in ]1, +\infty[)
\label{3}
\end{equation}
and at the same time it is not a polynomial of order at most $2$. Such mappings do exist and can be chosen to be continuously differentiable, see e.g.~the monograph of Kuczma \cite[Lemma 12.1]{Kuc68}. 
In fact, there exist a subinterval $I_0\subset ]1,+\infty[$ such that every map $\ph_0$ defined on $I_0$ can be extended to a solution of \eqref{3} defined on $]1, +\infty[$. 

Next, define $d\colon ]e, +\infty[ \to \mathbb{R}$ by the formula
\[
d(x) = x \cdot \ph(\ln (x)) \qquad (x \in ]e, +\infty[).
\]
One can check that we have 
\[
d(x^3) = 3x d(x) - 3x^2d(x^2) \qquad (x \in ]e, \infty[).
\]
Define operator $D\colon \mathscr{C}^1(]e, +\infty[)\to \mathscr{C}(]e, +\infty[)$ as 
\[
D(f)(x) = d \circ f(x)= d(f(x))  \qquad \left(f \in \mathscr{C}^1(]e, +\infty[)\right).
\]
A direct calculation shows that $D$ satisfies \eqref{id_single} for all $f \in \mathscr{C}(]e, +\infty[)$. On the other hand, $\ph$ is not a polynomial of order at most $2$. Therefore, there exist some $x^{\ast}, y^{\ast}, z^{\ast} \in ]e, \infty[$ such that
\[
\ph(x^{\ast}+y^{\ast}+z^{\ast}) - \ph(x^{\ast}+y^{\ast})-\ph(x^{\ast}+z^{\ast})-\ph(y^{\ast}+z^{\ast}) + \ph(x^{\ast})+\ph(y^{\ast}) + \ph(z^{\ast})\neq 0.
\]
Consequently, 
\[
d(x^{\ast}y^{\ast}z^{\ast}) - x^{\ast}d(y^{\ast}z^{\ast})-y^{\ast}d(x^{\ast}z^{\ast})-z^{\ast}d(x^{\ast}y^{\ast}) + x^{\ast}y^{\ast}d(z^{\ast})+x^{\ast}z^{\ast}d(y^{\ast}) + y^{\ast}z^{\ast}d(x^{\ast})\neq 0.
\]
Therefore, if we consider the functions 
\[
 f(x)= x^{\ast} \qquad g(x)= y^{\ast} 
 \qquad \text{and} \qquad 
 h(x)=z^{\ast} 
 \qquad 
 \left(x\in ]e, +\infty[\right), 
\]
then we conclude that $D$ does not satisfy \eqref{id_2}. 
\end{ex}

\begin{ackn}
The research of E.~Gselmann was supported by the János Bolyai Research Scholarship of the Hungarian Academy of Sciences.
\end{ackn}


%

\end{document}